\documentclass{amsart}

\usepackage[utf8]{inputenc}
\usepackage[english]{babel}

\usepackage[a4paper]{geometry}

\usepackage{amsmath}
\usepackage{amssymb}
\usepackage{amsfonts}
\usepackage{physics}
\usepackage{amsthm}
\usepackage{graphicx}
\usepackage{mathtools}
\usepackage{cases}
\usepackage{hyperref}
\usepackage{enumerate}
\usepackage{xspace}

\newtheorem{theorem}{Theorem}[section]

\newtheorem{lemma}[theorem]{Lemma}

\theoremstyle{definition}

\theoremstyle{remark}
\newtheorem{remark}{Remark}

\newcommand{\C}{\mathbb{C}}
\newcommand{\N}{\mathbb{N}}

\newcommand{\eqd}{=}
\newcommand{\V}{\textsc{vdm}}

\newcommand*{\Num}{N\textsuperscript{o}\xspace}

\begin{document}
\title{ A new estimate of the transfinite diameter of Bernstein sets}
\author{Dimitri Jordan Kenne}
\address{Doctoral school of Exact and Natural Sciences, Jagiellonian University, Łojasiewicza 11, Kraków, 30-348, Lesser Poland, Poland.}
\email{dimitri.kenne@doctoral.uj.edu.pl}
\date{\today}

\begin{abstract}
Let $K \subset \mathbb{C}^n$ be a compact set satisfying the following Bernstein inequality: for any $m \in \{ 1,..., n\}$ and for any $n$-variate polynomial $P$ of degree $\deg(P)$ we have 
\begin{align*}
    \max_{z\in K}\left|\frac{\partial P}{\partial z_m}(z)\right| \le M \deg(P) \max_{z\in K}|P(z)| \ \mbox{ for } z = (z_1, \dots, z_n). 
\end{align*}
for some constant $M= M(K)>0$ depending only on $K$. We show that the transfinite diameter of $K$, denoted $\delta(K)$, verifies the following lower estimate 
\begin{align*}
    \delta(K) \ge \frac{1}{n M}, 
\end{align*}
which is optimal in the one-dimensional case. In addition, we show that if $K$ is a Cartesian product of compact planar sets then 
\begin{align*}
    \delta(K) \ge \frac{1}{M}. 
\end{align*}
\end{abstract}

\subjclass[2000]{Primary 41A17, 31C15, 32U15}

\keywords{Transfinite diameter, Bernstein inequality, Markov inequality}

\maketitle

\section{Introduction}\label{sec1}
Sets verifying a Markov inequality are of particular interest in approximation and pluripotential theories.  In \cite{plesniak1990markov}, Ple\'sniak posed the question of the continuity of Green pluricomplex function for these sets. In particular, he wanted to know whether they are nonpluripolar or not, which is equivalent to verifying that the transfinite diameter is positive or zero (see \cite{levenberg1984}).  A subclass of these Markov sets known as Bernstein sets, have been proven by Siciak in \cite{siciak1997wiener} to be nonpluripolar. Some positive lower bounds of the transfinite diameter of Bernstein sets were proven by Białas-Cież and Jedrzejowski (first) in \cite{bialas2002transfinite} and Yazici (later) in \cite{yazici2022note}; therefore emphasing their nonpluripolarity. In this paper, we present a better lower estimate of the transfinite diameter of Bernstein sets which is optimal in the one-dimensional case. \par
We consider $\mathcal{P}_d(\C^n)$, the space of $n$-variate polynomials of total degree at most $d$. A compact set $K \in \C^n$ is a \textbf{Markov set} if it satisfies the following Markov inequality: for each $m \in \{1, \dots, n\}$ and for every polynomial $P \in \mathcal{P}_d(C^n)$
\begin{equation}
    \max_{z\in K}\big|\frac{\partial P}{\partial z_m}(z) \big| \le M d^r \max_{z\in K}|P(z)|\quad \mbox{ for }  z = (z_1, \dots, z_n)\label{equ1} 
\end{equation}
 for some constants $M= M(K)>0$, $r=r(K)>0$ which depend only on $K$.
 If \eqref{equ1} is verified with $r=1$ then $K$ is called a \textbf{Bernstein set}. For instance, the closed unit disc $D(0,1) = \{z \in \C: \ |z| \le 1\}$ is a well-known Bernstein set as it satisfies the inequality 
\begin{equation*}
    \max_{z\in D(0,1)}\left| P'(z) \right| \le  d \max_{z\in D(0,1)}|P(z)| 
\end{equation*} for all polynomial $P \in \mathcal{P}_d(\C)$. More generally, any finite union of $\mathcal{C}^2$-smooth Jordan curves is also a 
 Bernstein set (see \cite{nagy2005sharpening}). Examples of Bernstein sets with an infinite number of connected components have been constructed in \cite{tookos2005markov}.\par
The Bernstein's property is strongly related to the smoothness of Green pluricomplex function. Indeed, Siciak shows in \cite{siciak1997wiener} that, a set is Bernstein (or verifies a Bernstein inequality) if and only if its Green pluricomplex function is H\"older continuous (with exponent $\mu = 1$). It follows from this equivalence that Bernstein sets are nonpluripolar. As far as we know, the generalisation of this result to all Markov sets remains an open question (for $n \ge 2$).\par
We give in the preliminary section \ref{sec2}, the definition of the transfinite diameter of a given compact set $K \subset \C^n$ and throughout this paper we denote it, $\delta(K)$. Now, we recall the lower estimates of the transfinite diameter of Bernstein sets that have been proven so far. In \cite{bialas1998markov}, Bia{\l}as-Cie{\.z} established that the transfinite diameter of any Markov set $K\subset \C$  verifies the following estimate
\begin{equation}
    \delta(K) > \dfrac{1}{M}\sigma^r (\mbox{diam}(K))^{\frac{1}{3}}
\end{equation}
where $\sigma$ is an absolute positive constant and $\mbox{diam}(K)$ is the diameter of $K$. As a consequence, all planar Markov sets are nonpolar. Later, Bialas-Cie{\.z} and Jedrzejowski proved in  \cite{bialas2002transfinite} that any Bernstein set $K \subset \C^n$ satisfies 
\begin{equation}
    \delta(K) > \dfrac{1}{M2^{n-1}}. \label{equ4}
\end{equation}
Recently in \cite{yazici2022note}, Yazici improved the previous lower estimate for the transfinite diameter of any Bernstein set $K \subset \C^n $ for $n> 4$, by proving that 
\begin{equation}
    \delta(K) > \dfrac{1}{enM} \label{equ5}.
\end{equation}
In this paper, we use an idea from \cite{jkedrzejowski1992transfinite} involving generalized Leja points to obtain a better lower bound for the transfinite diameter of Bernstein sets. Our main result is the following.
\begin{theorem}\label{thm1}
Let $K$ be a Bernstein set in $\C^n$ of parameter $M$. Then
\begin{equation}
    \delta(K) \ge \frac{1}{n M}. \label{equ3}
\end{equation}
\end{theorem}
\begin{remark}
Note that the estimate in \eqref{equ3} is optimal in the one-dimensional case since the closed disc $D(0,R)= \{z \in \C: \ |z| \le R\}$ ($R>0$) is a Bernstein set with $M=1/R$ and its transfinite diameter is $\delta(D(0,R)) = R$. Moreover, the estimate in \eqref{equ3} is better than \eqref{equ5} for all natural numbers $n$ and better than \eqref{equ4} for all $n\ge 3$.
\end{remark}
It turns out that the estimate given in Theorem \ref{thm1} can be improved for the case of a Cartesian product of planar sets. 
\begin{theorem}\label{thm2}
Let $K=K_1 \times \dots \times K_n$ be a Bernstein set in $\C^n$ of parameter $M$. Then
\begin{equation}
    \delta(K) \ge \frac{1}{M}. \label{equ7}
\end{equation}
\end{theorem}

\section{Preliminaries} \label{sec2}
Let $\N = \{0,1,2, \dots\}$ be the set of natural numbers and let $\alpha: \N \ni j \longmapsto (\alpha_1(j), \dots, \alpha_n(j))$ be the enumeration on $\N^n$ associated with the graded lexicographical ordering which we denote by ``$\prec$". We write $|\beta| = \beta_1+ \dots +\beta_n$ for the length of a multi-index $\beta = (\beta_1, \dots, \beta_n)$.
Consider the monomials
\begin{align}
    e_j(z):= z^{\alpha(j)}  = z_1^{\alpha_1(j)} \cdots z_2^{\alpha_n(j)}, \quad j=1, 2, \dots.
\end{align}
The space of holomorphic polynomials of $n \ge 1$ complex variables and of degree at most $d \in \N$ is
\begin{equation}
     \mathcal{P}_d(\C^n) = \mbox{span}\{e_i:= z^{\alpha(i)} = z_1^{\alpha_1(i)}  \cdots 
     z_n^{\alpha_n(i)}; \  i \in \N \mbox{ and }\deg(e_i)\le d\}
\end{equation}
and its dimension is $h_d := \dim(\mathcal{P}_d(\C^n)) = \binom{n+d}{n}$.  For any set of points $\{\xi_0, \dots, \xi_{k-1}\}$ of $\C^n$, we define the generalized Vandermonde determinant $\V(\xi_0, \dots, \xi_{k-1})$ by:
\begin{equation}
    \V(\xi_0, \dots, \xi_{k-1}) := \det [e_i(\xi_j)]_{i,j = 0,1,2, \dots, k-1} 
\end{equation} with the convention $\V(\xi_0) := 1$.
For any $k \in \N$, we consider the constants 
\begin{equation}
     V_k := \max_{\xi_0, \dots, \xi_{k-1}\in K} \left| \V(\xi_0, \dots, \xi_{k-1}) \right|, \quad l_d := \sum_{i=1}^d i (h_i - h_{i-1}).
\end{equation}
The constant $l_d$ is the total degree of $\V(\xi_0, \dots, \xi_{h_d-1})$ viewed as a polynomial in $\xi_0, \dots, \xi_{h_d-1}$. It is known that $l_d = n\binom{n+d}{n+1}$. \par
The \textbf{transfinite diameter} of a compact set $K \subset \C^n$ is the constant
    \begin{equation}
        \delta(K) \eqd \limsup_{d \to + \infty} \delta_d(K),
    \end{equation}
    where $\displaystyle \delta_d(K) \eqd  V_{h_d}^{1/l_d}$ is called the $d$-th order transfinite diameter of $K$.
Fekete proved in \cite{fekete1923verteilung} that the limit $\delta(K)$ exists for any compact set $K \subset \C$ (i.e when $n=1$). Later in \cite{leja1959problemes}, Leja introduced the name ``transfinite diameter" and thus posed the problem of its existence when $n\ge 2$. A positive answer to his problem was given later by  Zaharjuta in \cite{zaharjuta1975transfinite}.\\
The transfinite diameter of a compact set can also be determined using Leja sequences. A Leja sequence for a compact set $K\subset \C^n$ is a sequence $(\xi_j)_{j\ge 0}$ such that $\xi_0$ is any arbitrary point (preferably on $\partial K$) and for each $N \ge 1$
          \begin{equation}
         |\V(\xi_0, \dots,\xi_{N-1}, \xi_N)| = \sup_{z \in K} |\V(\xi_0, \dots, \xi_{N-1},z)|.
     \end{equation}
It is proved in \cite{jkedrzejowski1992transfinite} that, for any compact set $K \subset \C^n$ 
\begin{equation}
    \delta(K) = \lim_{d \to +\infty}|\V(\xi_0, \dots,\xi_{h_d-1})|^{1/l_d}
\end{equation} 
We Consider the following mappings
\begin{align}
    \begin{array}{rccl}
         \partial_i:& \N^n & \longrightarrow & \N^n  \\
        & \beta & \longmapsto & \partial_i \beta = \begin{cases}
        (0, \dots, 0)& \text{if }\beta_i = 0\\
        (\beta_1, \dots , \beta_{i-1}, \beta_i -1, \beta_{i+1}, \dots, \beta_n)& \text{if } \beta_i \ge 1
        \end{cases},
    \end{array} 
\end{align}
$i=1, \dots,n$.\par Observe that $\displaystyle \dfrac{\partial z^\beta}{\partial z_m} = \beta_mz^{\partial_m(\beta)}$ for any multi-index $\beta \in \N^n$ and for any $i=1, \dots,n$.\\
The following lemma is a direct consequence to the fact that the graded lexicographical order is translation invariant.
\begin{lemma}\label{lem2}
Let $\beta = (\beta_1, \dots,\beta_n)$ and $\gamma = (\gamma_1, \dots , \gamma_n)$ be two multi-indices in $\N^n$. If $\beta \prec \gamma$ then for any $i\in \{1, \dots,n\}$ such that $\gamma_i \ge 1$ we have $\partial_i \beta \prec \partial_i \gamma$ or $\partial_i \beta = \partial_i \gamma = (0, \dots, 0)$.
\end{lemma}

\begin{proof}
Let $\beta \prec \gamma$ and let $i \in \{1, \dots, n\}$ such that $\gamma_i \ge 1$. We distinguishes two cases.
\begin{itemize}
\item If $\beta_i = 0$ then $\partial_i \beta= (0, \dots,0)$. Hence, $\partial_i \beta = \partial_i \gamma = (0, \dots, 0)$ or $\partial_i \beta \prec \partial_i \gamma$.
\item Suppose that $\beta_i \ge 1$. We have $\partial_i \beta = (\beta_1, \dots , \beta_{i-1}, \beta_i -1, \beta_{i+1}, \dots, \beta_n)$\\ and $\partial_i \gamma = (\gamma_1, \dots , \gamma_{i-1}, \gamma_i -1, \gamma_{i+1}, \dots, \gamma_n)$.
There are two cases to consider:
\begin{itemize}
    \item If $|\beta| < |\gamma|$ then $|\partial_i \beta| = |\beta|-1 < |\gamma|-1 =|\partial_i \gamma|$. It follows directly from the definition of the graded lexicographical order that $\partial_i \beta \prec \partial_i\gamma$.
    \item If $|\beta| = |\gamma|$ then we also have $\partial_i \beta \prec \partial_i\gamma$ since the lexicographical order is translation invariant and by hypothesis $\beta \prec \gamma$.
\end{itemize}
\end{itemize}

\end{proof}

\section{Proofs of Theorem \ref{thm1} and Theorem \ref{thm2}}\label{sec3}
We begin by establishing two lemmas (Lemma \ref{lem1} and Lemma \ref{lem3}) which are obtained by using the Markov's property over the following class of polynomials
	\begin{equation}
		\mathcal{P}^i \eqd \{P_i(z) = e_i(z) + \sum_{0 \le j < i} c_j e_j(z): \ c_j \in \C\} \quad \text{for } i \in \N.
	\end{equation}
  It is essentially due to these results that we can obtain a better estimate of the transfinite diameter of Bernstein sets.
\begin{lemma}\label{lem1}
Let $K$ be a Markov set in $\C^n$ of parameters $(M,r)$. For every polynomial\\ $P_i(z) = e_i(z) + \sum_{0 \le j < i} c_j e_j(z)\in \mathcal{P}^i$, we have
\begin{equation}
    \max_{z \in K}|D^{\alpha(i)} P_i(z)| \le M^{|\alpha(i)|} (|\alpha(i)|!)^r\ \max_{z\in K} |P_i(z)| \label{equ2}
\end{equation}
where $D^\alpha P := \dfrac{\partial^{\abs{\alpha}}P}{\partial z_1^{\alpha_1} \cdots \partial z_n^{\alpha_n}}$.
\end{lemma}

\begin{proof}
 We give a proof by induction on $i \in \N$.
For $i=0$ the inequality \eqref{equ2} is easily verified since we have $D^{\alpha(0)} P_0 = D^{(0, \dots,0)} (1) = 1$.
Inductive step: fix $k\ge 1$ and suppose that \eqref{equ2} is true for all $0\le i\le k-1$. Let us show that it is also verified at the rank $i=k$. Let $d\ge 0$ such that $h_d \le k < h_{d+1}$. Since $|\alpha(k)|\ge 1$, we can select a certain $m \in \{1, \dots,n\}$ for which $\alpha_m(k)$ is positive. Let $k_0 \in \{0,\dots,k\}$ such that
\begin{equation*}
    \alpha(k_0) = \partial_m \alpha(k) = (\alpha_1(k), \dots , \alpha_{m-1}(k), \alpha_m(k) -1, \alpha_{m+1}(k), \dots, \alpha_n(k)).
\end{equation*}
We have $D^{\alpha(k)} P_{k} = D^{\alpha(k_0)}\left(\frac{\partial P_{k}}{\partial z_m}\right)$. Moreover, Lemma \ref{lem2} guarantees that we can write
\begin{equation*}
    \frac{\partial P_{k}}{\partial z_m} = \alpha_{m}(k)z^{\alpha(k_0)} + \sum_{j=0}^{k_0} b_j e_j = \alpha_{m}(k) \left(e_{k_0} + \sum_{j=0}^{k_0} \tilde{c}_j e_j \right).
\end{equation*}
for some constants $b_j, \tilde{c}_j \ge 0$ (eventually equal to zero), $j=1, \dots, k_0$.\\
By setting $Q_{k_0} = e_{k_0} + \sum_{j=0}^{k_0} \tilde{c}_j e_j$, it follows that $D^{\alpha(k)} P_{k} = \alpha_{m}(k) D^{\alpha(k_0)} Q_{k_0}$. Using the hypothesis of induction with the polynomial $Q_{k_0}$, we deduce that
\begin{align*}
    \max_{z\in K} |D^{\alpha(k)} P_{k}(z)|  &= \alpha_{m}(k) \max_{z\in K} |D^{\alpha(k_0)} Q_{k_0}(z)|  \le  \alpha_{m}(k)\ M^{|\alpha(k_0)|} (|\alpha(k_0)|!)^r \ \max_{z\in K}|Q_{k_0}(z)|\\
    & \le \left( \alpha_{m}(k)\ M^{|\alpha(k_0)|} (|\alpha(k_0)|!)^r\right) \left( \frac{1}{\alpha_{m}(k)}\max_{z\in K}\left|\frac{\partial P_{k}}{\partial z_m}(z) \right|\right)\\
    & = M^{|\alpha(k_0)|} (|\alpha(k_0)|!)^r\ \max_{z\in K}\left|\frac{\partial P_{k}}{\partial z_m}(z)\right|.
\end{align*}
Moreover, from the definition of Markov set we have
\begin{equation*}
    \max_{z\in K} \left|\frac{\partial P_{k}}{\partial z_m}(z)\right| \le M |\alpha(k)|^r \max_{z\in K}|P_{k}(z)|.
\end{equation*}
Therefore, $$\max_{z\in K}|D^{\alpha(k)} P_{k}(z)| \le M^{|\alpha(k_0)|+1} |\alpha(k)|^r (|\alpha(k_0)|!)^r\ \max_{z\in K}|P_{k}(z)| \le M^{|\alpha(k)|} (|\alpha(k)|!)^r\ \max_{z\in K}|P_{k}(z)| $$ since $|\alpha(k)| = |\alpha(k_0)| + 1$. Thus \eqref{equ2} is true for all $i \ge 0$.
\end{proof}
The following Lemma provides a better version of Inequality \eqref{equ2} for the case of Cartesian product sets.
\begin{lemma}\label{lem3}
Let $K=K_1 \times \dots \times K_n$ be a Markov set in $\C^n$ of parameters $(M,r)$. For every polynomial $P_i(z) = e_i(z) + \sum_{0 \le j < i} c_j e_j(z) \in \mathcal{P}^i$, we have
\begin{equation}
    \max_{z \in K}|D^{\alpha(i)} P_i(z)| \le M^{|\alpha(i)|} (\alpha(i)!)^r\ \max_{z\in K} |P_i(z)|, \label{equ6}
\end{equation}
where $\alpha! = \alpha_1!\cdots \alpha_n!$.
\end{lemma}

\begin{proof}
 The proof is similar to that of Lemma \ref{lem2}. We proceed by induction on $i \in \mathbb{N}$. The case $i=0$ is obvious. Now, fix $k\ge 1$, $P_k \in \mathcal{P}^k$ and suppose that the inequality \eqref{equ6} is true for all $0\le i\le k-1$. We can assume without loss of the generality that $\alpha_1(k) \ge 1$. Let us consider $k_0 \in \{0, \dots, k-1\}$ such that $\alpha(k_0) = \partial_1 \alpha(k)$.
 By setting $Q_{k_0} = e_{k_0} + \sum_{j=0}^{k_0} \tilde{c}_j e_j$ so that $D^{\alpha(k)} P_{k} = \alpha_{1}(k) D^{\alpha(k_0)} Q_{k_0}$, as in the proof of Lemma \ref{lem1}, we obtain by the hypothesis of induction 
 \begin{equation}
     \max_{z\in K} |D^{\alpha(k)} P_{k}(z)| \le M^{|\alpha(k_0)|} (\alpha(k_0)!)^r\ \max_{z\in K}\left|\frac{\partial P_{k}}{\partial z_1}(z)\right|. \label{equ8}
 \end{equation} 
 For $\xi_2,\dots, \xi_n \in \mathbb{C}$ fixed, the polynomial $P_{k}(z_1,\xi_2, \dots, \xi_n)$ belongs to $\mathcal{P}_{\alpha_1(k)}(\mathbb{C}) \subset \mathcal{P}_{\alpha_1(k)}(\mathbb{C}^n)$. Therefore, the Markov inequality gives
 \begin{equation*}
     \left|\frac{\partial P_{k}}{\partial z_1}(z_1,\xi_2, \dots, \xi_n)\right| \le M (\alpha_1(k))^r \max_{z\in K}|P_{k}(z)| \quad \text{for all } z_1 \in K_1.  
 \end{equation*}
 Since $\xi_2,\dots, \xi_n \in \mathbb{C}$ are arbitrarily chosen, it follows that
  \begin{equation*}
     \max_{z\in K} \left|\frac{\partial P_{k}}{\partial z_1}(z)\right| \le M (\alpha_1(k))^r \max_{z\in K}|P_{k}(z)|  
 \end{equation*}
 and hence the inequality \eqref{equ8} implies 
\begin{align*}
     \max_{z\in K} |D^{\alpha(k)} P_{k}(z)| &\le M^{\abs{\alpha(k_0)}} (\alpha(k_0)!)^r\ \left(M \alpha_1(k)^r \max_{z\in K}|P_{k}(z)|\right)\\
     & = M^{|\alpha(k)|} (\alpha(k)!)^r\ \max_{z\in K} |P_k(z)|. 
 \end{align*}
 because $\abs{\alpha(k_0)}+1 = \abs{\alpha(k)}$ and $\alpha(k_0)! = (\alpha_1(k)-1)! \alpha_2(k)!\cdots \alpha_n(k)!$.
\end{proof}

\begin{remark}
    We deduce from Lemma \ref{lem1} and Lemma \ref{lem3} that the inequalities \eqref{equ2} and \eqref{equ6} are valid for all polynomials $P_i(z) = \sum_{j =0}^i c_j e_j(z)$, with $c_j \in \mathbb{C}$. 
\end{remark}
\begin{remark}
It is interesting to note that we have the following result: for all compact sets $K_j \subset \mathbb{C}^{n_j}$ ($n_j \in \mathbb{N}$), $j=1,\dots,m$
        \begin{equation}
            K_1 \times \cdots \times K_m \subset \mathbb{C}^n \mbox{ is a Markov set } \Longleftrightarrow K_1, \dots, K_m \ \mbox{ are all Markov sets}.
        \end{equation}
        Indeed, if $K_1 \times \cdots \times K_m$ is a Markov set of parameters $(M,r)$ then using the fact that for each $d\in \mathbb{N}$ and for each $j =1, \dots, m$, $\mathcal{P}_d(\mathbb{C}^{n_j}) \subset \mathcal{P}_d(\mathbb{C}^{n})$, we can easily deduce that $K_1, \dots, K_m$ are also Markov sets of parameters $(M,r)$. On the other hand, if we suppose that $K_1,\dots, K_m$ are Markov sets with respective parameters $(M_1,r_1), \dots, (M_m,r_m)$, it is straightforward to see that  $K_1 \times \cdots \times K_m$ is a Markov set of parameters ($ \displaystyle \max_{j=1, \dots, m} M_j, \max_{j=1, \dots, m} r_j$). 
\end{remark}
Now we can prove Theorem \ref{thm1} and Theorem \ref{thm2}.
\begin{proof}[Proof of Theorem \ref{thm1}]
Let $K$ be a Bernstein set in $\C^n$ and let $(\xi_i)_{i \ge 1}$ be a Leja sequence for $K$. Then $K$ is necessarily determining for the space of polynomial as a Bernstein set, i.e $P \in \bigcup_{d\ge 0}\mathcal{P}_d(\C^n)$ and $P \equiv 0$ on $K$ imply $P \equiv 0$ in $\C^n$. Therefore, all Leja sequences for $K$ are unisolvent, i.e 
\begin{equation*}
    \V^{(n)}(\xi_0, \dots, \xi_{i-1}) \not = 0 \quad \mbox{for all $i \ge 1$}.
\end{equation*} 
 Set
\begin{equation*}
    P_i(\xi) : = \dfrac{ \V(\xi_0, \dots, \xi_{i-1}, \xi)}{\V(\xi_0, \dots, \xi_{i-1})} \quad \mbox{ for all $i \ge 1$}.
\end{equation*}
By expanding the Vandermonde determinant, which is in the numerator of $P_i$, following its last column we obtain for each $i\ge 1$
\begin{equation*}
    P_i(\xi) = e_{i}(\xi) + \sum_{0\le j< i} c_j e_j(\xi).
\end{equation*} for some constants $c_i \in \C$.
Now observe that $\V(\xi_0, \dots, \xi_{i-1}, \xi_i) = \prod_{j=1}^i P_j(\xi_j)$, $i\ge 1$. Hence, using the definition of Leja sequence and then Lemma \ref{lem1} we obtain
\begin{align*}
    |\V(\xi_0, \dots, \xi_{h_d-1})| & = \left| \prod_{i=1}^{h_d-1} P_i(\xi_i) \right| = \prod_{i=1}^{h_d-1} \max_{z\in K}\abs{P_i(z)}\\  &\ge \prod_{i=1}^{h_d-1} \left[\frac{1}{ M^{|\alpha(i)|} |\alpha(i)|!} \norm{D^{\alpha(i)} P_i}_K \right] = \prod_{i=1}^{h_d-1} \left[\frac{\alpha(i)!}{ M^{|\alpha(i)|} |\alpha(i)|!}  \right] \\ 
    &\ge \prod_{i=1}^{h_d-1} \left[\frac{1}{ (nM)^{|\alpha(i)|} }  \right] = \frac{1}{ (nM)^{l_d}}.
\end{align*}
The last inequality is due to the fact that $D^{\alpha(i)} P_i = \alpha(i)! \ge \frac{1}{n^{|\alpha(i)|}} |\alpha(i)|!$ and $l_d=\sum_{i=1}^{h_d-1} |\alpha(i)|$. It follows that 
\begin{equation*}
    \delta(K) = \lim_{d \to +\infty} |\V(\xi_0, \dots, \xi_{h_d-1})|^{1/l_d} \ge \frac{1}{nM}
\end{equation*}
which is the desired estimate.
\end{proof}

\begin{proof}[Proof of Theorem \ref{thm2}]
    The proof is the same as that of Theorem \ref{thm1} but here we use Lemma \ref{lem3} to replace Lemma \ref{lem1}.
\end{proof}

\section*{Acknowledgments}
I wish to thank Professor Leokadia Bia{\l}as-Cie{\.z} for her valuable comments on this work.\\
Funding: This work was partially supported by the National Science Center, Poland, grant Preludium Bis 1 \Num 2019/35/O/ST1/02245

\bibliographystyle{plain}

\end{document}